\documentclass{amsart}
\usepackage[utf8]{inputenc}
\usepackage{amsmath}
\usepackage{amssymb}
\usepackage{enumitem}
\usepackage{mathrsfs}
\usepackage{mathtools}
\usepackage{amsrefs}
\usepackage{stmaryrd} 
\usepackage[colorlinks]{hyperref}

\theoremstyle{plain} 

\newtheorem{thm}{Theorem}
\newtheorem*{mainthm}{Main theorem}

\newtheorem{lemma}[thm]{Lemma}

\newtheorem{coro}[thm]{Corollary}

\theoremstyle{definition}
\newtheorem{remark}[thm]{Remark}

\newtheorem{definition}[thm]{Definition}

\newcommand{\Z}{\mathbb{Z}}
\newcommand{\N}{\mathbb{N}}
\newcommand{\SL}{\mathrm{SL}}
\newcommand{\GL}{\mathrm{GL}}
\newcommand{\Aut}{\mathrm{Aut}}
\newcommand{\Alt}{\mathrm{Alt}}

\newcommand\operad{{\mathcal O}}
\newcommand{\Id}{\mathsf{Id}}
\newcommand{\K}{K}

\newcommand{\eval}{\mathsf{ev}}
\newcommand{\ar}[1]{{\mathsf{ar}}(#1)}

\newcommand{\question}[1]{{\textcolor{red}{#1}}}

\newcommand{\mk}[1]{\textcolor{red}{MK: #1}}

\begin{document}
\title{Groups with property (T) and  many alternating group quotients}

\author[L.~Bartholdi]{Laurent Bartholdi}
\address{Saarland University, Germany}
\thanks{L.B. is supported in part by the ERC AdG grant 101097307}

\author[M.~Kassabov]{Martin Kassabov}
\address{Cornell University, USA}
\thanks{M.K. is supported in part by the Simon's Foundation grant 713557}

\title{Property (T) and Many Quotients}
\date{21 August 2023}
\begin{abstract}
    We prove that, for the free algebra over a sufficiently rich operad $\mathcal O$, a large subgroup of its group of tame automorphisms has Kazhdan's property (T).
    
    We deduce that there exists a group with property (T) that maps onto large powers of alternating groups.
\end{abstract}
\maketitle

\section{Introduction}
Property (T), introduced by Kazhdan in~\cite{kazhdan:T}, may be thought of a strong, analytic form of finite generation --- it remains the most direct path to proving that lattices in higher rank Lie groups are finitely generated. In this spirit, one should expect groups with this property to have tight restrictions on their quotients. In particular, the Cayley graphs of their quotients form so-called \emph{expander graphs}, characterized for example by a spectral gap in their combinatorial Laplacian. Conversely, families of expander graphs are conveniently ``explained'' by their being Cayley graphs of quotients of a single group with property (T). The performance of the ``product replacement algorithm'' (producing almost uniform samplings of black box groups) is thus explained by property (T) of automorphism groups of free groups~\cite{kno:autf5}.

It remained for long an open question whether there exists a group with property (T) that admits all (or at least infinitely many) alternating groups as quotients; see e.g.~\cite{lubotzky:discrete}*{Problems~10.3.2--10.3.4}. A hint that this might be possible appears in~\cite{kassabov:alt}, proving that alternating groups admit generating sets turning them into expander graphs. The question was settled in~\cite{kno:autf5}: Kaluba, Nowak and Ozawa prove that $\Aut(F_5)$ has property (T), while Gilman had previously shown that $\Aut(F_5)$ maps onto infinitely many alternating groups. Another example of (T) group mapping onto infinitely many alternating groups appears in~\cite{caprace-kassabov:t}, while a $(\tau)$ group (see below) mapping onto all alternating groups appears earlier in~\cite{ershov-jaikin-kassabov:t}*{Theorem~9.17}.

In a completely different direction, Philip Hall proved that, for every finite simple group $G$, the minimal number of generators of the direct power $G^n$ grows logarithmically in $n$, and may remain bounded for unbounded exponents. In fact, for every $k\ge2$ there exists a superexponential function $f_k(n)$ such that $\Alt(n)^{\times f_k(n)}$ is $k$-generated. For instance, for $\Alt(5)$ the alternating group on $5$ letters, $\Alt(5)^{\times 19}$ is $2$-generated while $\Alt(5)^{\times 20}$ requires $3$ generators; see~\cite{hall:eulerian}*{\S1.6}.

It is thus theoretically possible that there exists a group with property (T) and mapping onto large powers of alternating groups. We achieve a result of this kind in this article, with a superpolynomial exponent:
\begin{mainthm}[= Corollary~\ref{coro:main}]
    For every $d\ge2$ there is a group $\Gamma_d$ with property (T) that surjects onto $\Alt(p^{(d+2)k} -1)^{\times p^{k^{d+1}} }$ for all primes $p > 3+d+4\sqrt{d-1}$ and all $k \geq 1$.
\end{mainthm}

This implies that for every $d$ there is a group with property (T) that surjects onto $\Alt(n)^{\times f(n)}$ for infinitely many $n$, with $f(n) \approx \exp (\log n)^d$. We leave as an open question whether there exists a group with property (T) that surjects onto $\Alt(n)^{\times f(n)}$ for infinitely many $n$, with $f(n) \approx \exp (n)$, or at least $f(n)\succsim \exp(n^\alpha)$ for some $\alpha > 0$.

Prior to this work, it could be derived, using the techniques from~\cite{kassabov:alt} and some extra work, that the Cayley graphs of $\Alt(n)^{\times n^k}$ may form a family of expander graphs of bounded degree, for arbitrary $k\in\N$; and using~\cite{ershov-jaikin-kassabov:t} that there exists a group with property ($\tau$) mapping onto all $\Alt(n)^{\times n^k}$.

\subsection{Property (T)}
We recall only very briefly the definition of property (T), for details see e.g.~\cite{bekka-harpe-valette:t}. A discrete group $\Gamma$ has \emph{property (T)} if its trivial representation is isolated within unitary representations; this means the following. A representation of $\Gamma$ on a Hilbert space $\mathscr H$ \emph{almost has invariant vectors} if for every $\epsilon>0$ and every finite $S\subseteq\Gamma$ there is $v\in\mathscr H$ with $\|s v-v\|<\epsilon\|v\|$ for all $s\in S$. The group $\Gamma$ has property (T) if every unitary representation of $\Gamma$ that almost has invariant vectors actually has $\Gamma$-invariant vectors.

We shall also make use of a relative version of (T): the pair $(\Gamma,H)$ with $H\le\Gamma$ has \emph{relative property (T)} if every unitary $\Gamma$-representation that almost has invariant vectors has $H$-invariant vectors. Thus, $\Gamma$ has property (T) if and only if $(\Gamma,\Gamma)$ has relative property (T), and if $H$ is finite then $(\Gamma,H)$ always has relative property (T). This relative property appeared often as a stepping stone in the proof that a group has property (T), and this article is no exception. Briefly, if $(\Gamma,H_i)$ have property (T) for a collection of subgroups $(H_i)$, the invariant subspaces of the $H_i$ have well-controlled, large enough angles, and $G$ is generated by $\bigcup_i H_i$, then $\Gamma$ has (T), see~\cite{kassabov:t}. We shall review more precisely the required condition in the course of the proofs.

Margulis realized in~\cite{margulis:expander} that property (T) leads to explicit constructions of expander graphs, namely the Cayley graphs of finite quotients. A weaker property, called \emph{property $(\tau)$} by Lubotzky, already yields this conclusion: it suffices that the trivial representation be isolated among those representations of $\Gamma$ that factor through a finite quotient of $\Gamma$.

\subsection{Acknowledgments}
The second author is grateful to the organizers, and participants of the workshop on profinite rigidity at ICMAT in June 2023 for the stimulating discussions.

\section{Universal algebra}
Let $\operad$ be an \emph{operad}: a collection $(\operad(n))_{n\in\N}$ of abstract operations, with compositions $\circ_i\colon\operad(m)\times\operad(n)\to\operad(m+n-1)$ for $1\le i\le n$, thought of a composing an arity-$m$ operation with an arity-$m$ operation by feeding the output of the former as $i$th input of the latter. These composition maps obey the obvious associativity law. There is also an action of the symmetric group on $\operad(n)$, permuting each operation's inputs, and compatible with the compositions.

The operad $\operad$ is \emph{generated} by the set $S$ if each operation in $\operad(n)$ can be obtained as a composition of operations in $S$. For simplicity we will not allow operations of arity zero (i.e., constants) in our operads; however we do not put any other restrictions, so for example operations of arity $1$ (namely maps) are allowed.

We shall not need much from the theory of operads, so we concentrate immediately on a special case that serves our purposes: the \emph{free operad} on a finite graded set $S$, which can be defined by the usual universal property, and also has the following concrete description.
Let $S$ be a finite set of abstract operations $\{\star_s : s\in S\}$, each with its \emph{arity} $\ar{s}\in\N$. We denote by $\ar{S}=\max\{\ar{s}:s\in S\}$ the maximal arity of $S$. The \emph{free operad} $\operad_S$ on $S$ consists of all compositions of operations in $S$, with an ordering of their inputs. The elements of $\operad_S(n)$ are rooted trees with $n$ leaves numbered $1,\dots,n$, with at each non-leaf vertex a label $s\in S$ and $\ar{s}$ descendants in a given order.

\begin{definition}
    Let $R$ be a commutative ring, and let $\operad$ be an operad. An \emph{$\operad$-algebra over $R$} is an $R$-module $A$ endowed with a family of $R$-multilinear maps $A^n\to A$, one for each element of $\operad(n)$, satisfying the usual operad axioms.
    
    If $\operad$ is the free operad on $S$, this is equivalent to being given a family of $R$-multilinear maps $\star_s\colon A^{\ar{s}}\to A$, one for each $s\in S$.
\end{definition}

In the category of $\operad$-algebras over $R$ there is a free object on any set $X$ which we denote by $R\langle X \rangle_\operad$. As an $R$-module, it is generated by all rooted trees of height $1$ with leaves labeled by $X$ and an element in $\operad(n)$ labelling the root.

Note that $R\langle X \rangle_\operad$ has a natural grading in which all variables have degree $1$ and all operations have degree $0$. If $X$ is finitely generated and $\operad$ is finitely generated and does not contain any operations of arity $1$ then the homogeneous components of $R\langle X \rangle_\operad$ are finitely generated $R$-modules. In the case of a free operad $\operad_S$ the free algebra $R\langle X \rangle_{\operad_S}$ has homogeneous components of arbitrarily large degrees provided that $S$ contains at least one operation of arity $\geq 2$.

\section{Tame automorphisms}
Let $\operad$ be an operad generated by a finite set $S$ of operations as in the previous section. Consider a commutative ring $R$, and let $F=R\langle X\rangle_{\operad}$ denote the free $\operad$-algebra on $X=\{x_0,\dots,x_{n-1}\}$. We shall consider a certain subgroup $\Gamma_{n,N,\operad}$ of group of \emph{tame} automorphisms of $F$.

For $0\le i<n$ and $f\in R\langle x_0,\dots,\widehat{x_i},\dots,x_{n-1}\rangle_{\operad}$, consider the \emph{transvection} 
\[
    t_i(f)\colon F\to F,\quad t_i(f)(x_j)=\begin{cases}x_j+f & \text{ if }i=j,\\ x_j & \text{otherwise}.\end{cases}
\]
Evidently $t_i(f)$ is an automorphism of $F$, with inverse $t_i(-f)$. 
By definition, the group of \emph{tame automorphisms} of $F$ is the group generated by all such transvections.

\begin{definition}
    Let $S$ be a generating set of an operad $\operad$, and choose $n>\max\{\ar{S},2\}$. Let $N\in\N$ be any, and consider the ring $R=\Z[1/N]$. The group $\Gamma_{n,N,\operad}$ is defined as the subgroup of $\Aut(F)$ generated by\footnote{By its construction, the group depends not only on the operad $\operad$ but also on the choice of the generating set $S$. This dependence is very mild as we will show in Theorem~\ref{thm:transvections}, and is not reflected in the notation.}
    \begin{align*}
        \alpha_i &\coloneqq t_{i-1}(x_i)\text{ for }1\le i<n,\\
        \alpha_n &\coloneqq t_{n-1}(x_0/N),\\
        \beta_s &\coloneqq t_0(\star_s(x_1,\dots,x_{\ar{s}}))\text{ for }s\in S.
    \end{align*}
    For brevity we write it simply $\Gamma_n$ when the dependency on $N$ and $\operad$ is irrelevant.
\end{definition}

Our first main result is that the group $\Gamma_n$ has property (T) as soon as the parameter $N$ is large enough:
\begin{thm}\label{thm:t}
    If $N$ is divisible by all primes $p\le 3+\ar{S}+4\sqrt{\ar{S}-1}$, then $\Gamma_{n,N,\operad}$ has Kazhdan's property (T).
\end{thm}
\begin{proof}
    Notice first that the automorphisms $\alpha_1,\dots,\alpha_n$ generate $\SL_n(R)$, and recall that $\SL_n(R)$ has property (T) since $n\ge3$.

    There exists therefore a constant $\delta$ such that, for any representation of $\Gamma_n$ on a Hilbert space $\mathscr H$, any $\epsilon>0$ and any vector $v\in\mathscr H$ which is $\epsilon$-almost invariant under the action of the generators of $\Gamma_n$, we have that $v$ is $\delta\epsilon$-almost invariant under $\SL_n(R)$, and in particular under all $t_i(r x_j)$ with $0\le i\neq j<n$ and $r\in R$, and also under all $t_0(r\star_s(x_1,\dots,x_{\ar{s}}))$ with $r\in R$ and $s\in S$ since these are words of bounded length in $\SL_n(R)$ and the generators of $\Gamma_n$.

    Consider the following abelian subgroups of $\Gamma_n$:
    \begin{align*}
        G_0 &= \langle t_0(r x_1),t_0(r\star_s(x_1,\dots,x_{\ar{s}})):r\in R,s\in S\rangle,\\
        G_i &= \{t_i(r x_{i+1\bmod n}) : r\in R\}\text{ for }1\le i<n.
    \end{align*}
    Then by the previous paragraph the pairs $(\Gamma_n,G_i)$ all have relative property (T).

    For all $i<j$, the group generated by $G_i$ and $G_j$ is either abelian or nilpotent of class $2$: if $0<i<j-1$ then it is abelian, isomorphic to $R^2$; if $0<i=j-1$ then it is isomorphic to the Heisenberg group of upper-triangular $3\times 3$ matrices over $R$; and if $i=0$ then it is isomorphic to a subgroup of a product, over $s\in S$, of either $R^2$ (if $\ar{S}<j<n-1$) or the Heisenberg group (if $j\le \ar{S}$ or $j=n-1$).

    It follows that, in a representation as above, the Friedrichs angles between invariant subspaces for $G_i,G_j$ satisfy
    \[
        0\le\cos\sphericalangle(\mathscr H^{G_i},\mathscr H^{G_j})\le\begin{cases}0 & \text{ if $\langle G_i,G_j\rangle$ is abelian},\\ p^{-1/2} & \text{ otherwise},\end{cases}
    \]
    where $p$ is the smallest prime not dividing $N$. (Recall that the angle between two subspaces $V,W\le\mathscr H$ is the smallest angle between vectors in $V\cap(V\cap W)^\perp$ and $W\cap(W\cap V)^\perp$). Indeed, it suffices to consider representations of the Heisenberg group over $\mathbb Z/p$, which have dimension $1$ or $\ge p$; and then the bound on the angles is~\cite{ershov-jaikin:universal}*{Theorem~4.4}.
    
    To apply~\cite{kassabov:t}*{Theorem~1.2}, it remains to prove that the following matrix is positive definite:
    \[
        \Delta\coloneqq
        \begin{pmatrix}
            1 & -\varepsilon & \cdots & -\varepsilon & 0 & \cdots & 0 & -\varepsilon\\
            -\varepsilon & 1 & \ddots & & & & & 0\\
            \vdots & \ddots & \ddots & -\varepsilon & & & & \vdots\\
            -\varepsilon & 0 & -\varepsilon & 1 & -\varepsilon\\
            0 & & & -\varepsilon & 1 & \ddots & & \vdots\\
            \vdots & & & & \ddots & \ddots & -\varepsilon & 0\\
            0 & & & & & -\varepsilon & 1 & -\varepsilon\\
            -\varepsilon & 0 & \cdots & & \cdots & 0 & -\varepsilon & 1
        \end{pmatrix}
    \]
    for $\varepsilon = p^{-1/2}$, with terms `$-\varepsilon$' appearing one step away from the diagonal and in the first $\ar{S}$ entries of the first row and column.

    We can decompose $\Delta$ as the sum of a circulant matrix $\Delta_1$ with $2\varepsilon$ on the diagonal and $-\varepsilon$ off the diagonal, and a matrix $\Delta_2$ with $1-2\varepsilon$ on the diagonal and $-\varepsilon$ on the first $\ar{S}-1$ entries of the first row and column. The matrix $\Delta_1$ is positive semidefinite as soon when  $\varepsilon>0$, while $\Delta_2$ has eigenvalues $1-2\varepsilon$ (with multiplicity $n-2$) and $1 - 2\varepsilon \pm \sqrt{\ar{S}-1} \varepsilon$ (with multiplicity $1$). Thus $\Delta_2$  is positive definite when $\sqrt{\ar{S}-1} \varepsilon < 1 - 2\varepsilon$. We deduce that $\Delta$ is positive definite when $\varepsilon <1/(2+\sqrt{\ar{S}-1})$.
\end{proof}
\begin{remark}
    The bound for $p$ in Theorem~\ref{thm:t} is not optimal, but it can not be improved significantly. It can be shown that $\Delta$ is not positive definite when $\varepsilon >1/\max(2,\sqrt{\ar{S}-1})$, so the best bound is at least $\max(4, \ar{S}  - 1)$.
\end{remark}
Our next result is that the group $\Gamma_n$ contains a substantial part of the tame automorphism group of $F$, and really depends on $\operad$ (as the notation suggests) and only mildly on the choice of its generating set $S$:
\begin{thm}\label{thm:transvections}
    For all $f\in R\langle x_1,\dots x_{n-1-\ar{S}}\rangle_\operad$, the group $\Gamma_{n,N,\operad}$ contains the transvection $t_0(f)$.
\end{thm}
\begin{proof}
    By linearity, it is sufficient to prove this for elements of the free $S$-magma, namely for rooted tree. We proceed by induction on the tree's height, the case of single leafs being covered by the elementary matrices in $\SL_n(R)\subset\Gamma_{n,N,\operad}$.

    Consider therefore $s\in S$, write $k=\ar{s}$, and consider a term $f=\star_s(f_1,\dots,f_k)$ with $f_1,\dots,f_k\in R\langle x_1,\dots,x_\ell\rangle_\operad$ for some $\ell\in\N$ satisfying $k+\ell<n$. By induction, there are transvections $t_0(f_i)$ in $\Gamma_n$, and since $\Gamma_n$ contains all even permutations of the variables we may assume by induction that $\Gamma_n$ contains the transvections
    \[
        \gamma_j\coloneqq t_{\ell+j}(f_j)\text{ for }1\le j\le k.
    \]
    Note that the $\gamma_j$ all commute with each other. There is also in $\Gamma_n$ a conjugate $\beta_s'$ of $\beta$ that is the transvection
    \[
        \beta_s'\coloneqq t_0(\star_s(x_{\ell+1},\dots,x_{\ell+k})).
    \]
    By a direct computation,
    \begin{align*}
    [\beta_s',\gamma_j] &= (\beta_s')^{-1} (\beta_s')^{\gamma_j}
    = t_0(-{\star_s}(x_{\ell_1},\dots,x_{\ell+k}))\; t_0(\star_s(x_{\ell+1},\dots,x_{\ell+j}+f_j,\dots,x_{\ell+k}))\\
    &= t_0(\star_s(x_{\ell+1},\dots,f_j,\dots,x_{\ell+k}),
    \end{align*}
    so the iterated commutator $[\dots[\beta'_s,\gamma_1],\dots,\gamma_k]$ is the transvection $t_0(f)$ which thus belongs to $\Gamma$.
\end{proof}

\section{Representations}
For an $\operad$-algebra $A$ and $X=\{x_0,\dots,x_{n-1}\}$, consider the set $\mathscr R_{n,A}$ of $\operad$-algebra homomorphisms $R\langle X\rangle_\operad\to A$. Such a homomorphism is uniquely determined by the images of $x_0,\dots,x_{n-1}$, which are arbitrary elements of $A$ since $R\langle X\rangle_\operad$ is free. We may therefore naturally identify $\mathscr R_{n,A}$ with $A^n$.

The automorphism group of $R\langle X\rangle_\operad$ naturally acts on $\mathscr R_{n,A}$ by pre-composition. Under the identification of $\mathscr R_{n,A}$ with $A^n$, 
the generators $\alpha_i$ (for $1 \leq i < n$) act as $(a_0,\dots,a_{n-1})\mapsto(a_0,\dots,a_i-a_{i+1},\dots,a_{n-1})$, etc.

Furthermore, the action of $R\langle X\rangle_\operad$ commutes with the action of the automorphism group of $A$ by post-composition. Again choosing $R=\Z[1/N]$, we obtain an action of $\Gamma_{n,N,\operad}$ on $A^n/\Aut(A)$.

\begin{definition}
    An $\operad$-algebra $A$ over $R$ is called \emph{minimal} if its only subalgebras are $A$ and the $0$-submodule. Here by a subalgebra of $A$ we mean an $R$-submodule which is closed under all operad operations. 
\end{definition}

\begin{thm}\label{thm:hightransitive}
    Let $A$ be an $\operad$-algebra, and choose $n\ge \ar{S}+2$.
    \begin{enumerate}
        \item If $A$ is minimal and non-trivial then the action of $\Gamma_{n,N,\operad}$ on $A^n$ has two orbits: the fixed point $0^n$ and a large orbit consisting of all other points.
        \item If $A$ is minimal then the induced action of $\Gamma_n$ on $\Omega_{n,A}\coloneqq(A^n\setminus 0^n)/\Aut(A)$ is $k$-transitive, for all $k$ less than the number of $\Aut(A)$-orbits in $A$.
        In particular, if $A$ is finite and $|\Omega|> 25$ and $\Aut(A)$ has at least $6$ orbits on $A$ then $\Gamma_n$ acts on $\Omega$ as a full alternating or symmetric group.
        \item If $A,A'$ are two non-isomorphic minimal algebras then the actions of $\Gamma_n$ on $\Omega_{n,A}$ and on $\Omega_{n,A'}$ are not isomorphic.
    \end{enumerate}
\end{thm}

\noindent We begin by an analogue of the Chinese Remainder Theorem for minimal algebras:
\begin{lemma}\label{lem:crt}
    Let the $\operad$-algebra $A$ be minimal. For any elements $a_1,\dots,a_k\in A\setminus 0$ in distinct $\Aut(A)$-orbits and for every $b_1,\dots,b_k\in A$ there exists $v\in R\langle x\rangle_\operad$ such that the substitution $x \mapsto a_i$ maps $v$ to $b_i$, i.e., $\eval_{a_i}(v) = b_i$ where $\eval_{a}\colon R\langle x \rangle_\operad \to A$ is the evaluation map $x\mapsto a$.
\end{lemma}
\begin{proof}
    The proof is by induction on $k$. The base case $k=1$ follows from the minimality of $A$ which implies that the evaluation map $\eval_{a_1}$ is a  surjective map $R\langle x\rangle_\operad \to A$. Assuming the statement for $k$, the evaluation maps at $a_1,\dots, a_k$ yield a surjection 
    \[
        \eval_{a_1}\times\cdots\times\eval_{a_k} \colon R\langle x\rangle_\operad \to A^k.
    \] 
    The kernel $V$ of this map is a subalgebra of $R\langle x\rangle_\operad$ because $\operad$ has no constants; and the evaluation at $a_{k+1}$ maps $V$ to a subalgebra of $A$. Since $A$ is minimal, the image is either the whole of $A$, proving the induction step, or is $0$. In the last case $\eval_{a_{k+1}}$ is identically zero on $V$, so induces a (still surjective) algebra homomorphism $A^k \to A$. Pre-composing this homomorphism with the $i$th embedding $A\to A^k$ we obtain a homomorphism $\phi_i\colon A \to A$ mapping $a_i$ to $a_{k+1}$; and $\phi_i$ is non-zero so its kernel is $0$ and its image is $A$, i.e., $\phi_i$ is an automorphism of $A$; therefore $a_i$ and $a_{k+1}$ are in the same orbit of $\Aut(A)$.
\end{proof}

\begin{proof}[Proof of Theorem~\ref{thm:hightransitive}]
    (1) Consider $A$ a minimal algebra and $a\in A$ a non-zero element. Since $A$ is minimal, $a$ generates the whole algebra $A$, and we will show that the $\Gamma_n$-orbit of $(a,0,\dots,0)$ contains every non-zero element of $A^n$.

    Consider $(a_0,\dots, a_{n-1}) \in A^n \setminus 0^n$. Since $\Gamma$ contains the group of even permutations we can assume that $a_{n-1}\not=0$. Thus, each of $a_0-a, a_1, a_2,\dots,a_{n-1}$ may be respectively written as an expression $v_i(a_{n-1})$ since $a_{n-1}$ is non-zero and thus generates the algebra $A$. By Theorem~\ref{thm:transvections} and conjugation, the transvection $t_i(v_i)$ belongs to $\Gamma_n$ for all $1\le i<n$. Applying them in sequence, we see that $(a,0,\dots,a_{n-1})$ is in the same orbit as  $(a_0,a_1,\dots,a_{n-1})$. Finally, $a_{n-1}$ may be written as an expression in $a$ and another transvection from $\Gamma_n$ sends  $(a,0,\dots,a_{n-1})$ to  $(a,0,\dots,0)$.

    (2) For the second statement, we shall prove that the action of $\Gamma_n$ is $k$-transitive whenever $k$ is at most the number of $\Aut(A)$-orbits on $A\setminus 0$. Using Lemma~\ref{lem:crt}, the proof of $k$-transitivity is standard. Consider $a_1,\dots,a_k \in A \setminus 0$ in different orbits under $\Aut(A)$. Let $v_1,\dots,v_k$ be vectors in $A^n\setminus 0$ which are in different $\Aut(A)$-orbits under the diagonal action. We use induction on $k$ to show that there is an element in $\Gamma$ which sends $v_i$ to $(a_i,0,\dots, 0)$ for all $i=1,\dots,k$. The base case $k=1$ is the first statement of the theorem. For the induction step we can assume that $v_i=(a_i,0,\dots, 0) $ for $i=1,\dots, k$.
    If some coordinate $b_{k+1,j}$ of $v_{k+1}$ is non-zero for some $j>0$, then we can find a transvection which changes the zeroth coordinate of $v_{k+1}$ to $a_{k+1}$ and fixes $v_i$ for $i=1,\dots, k$ and then use Lemma~\ref{lem:crt} to move the resulting vector to $(a_{k+1},0,\dots, 0)$. Otherwise the zeroth coordinate of $v_{k+1}$ is in a different $\Aut(A)$-orbit than $a_1, \dots a_{m}$ and again by Lemma~\ref{lem:crt} we can find a transvection which fixes $v_1,\dots, v_k$ and makes some other coordinate of $v_{k+1}$ non-zero.

    The final claim in (2) follows from the well-known fact that there are no highly transitive groups acting on large finite sets except the alternating and the symmetric group.
  
    (3) For the last statement, let us assume that the actions of $\Gamma_n$ on $\Omega_{n,A}$ and on $\Omega_{n,A'}$ are isomorphic. Then, using the language of group theory, we can characterize the respective subsets $(A\setminus0)\times0^{n-1}$ and $(A\setminus0)\times0^{n-1}$ as the fixed sets of all transvections $t_i(a)$ with $1\le i<n$. The action of transvections $t_0(a)$ being isomorphic on these two sets then directly lets us reconstruct the $\operad$-algebra structure on $A,A'$ from the $\Gamma_n$-action.
\end{proof}
\begin{remark}
    It is likely that the minimality assumption on $A$ can be replaced with a weaker one, such as simplicity plus a small extra assumption (such as a bound on the number of generators of the subalgebras of $A$).
    This will slightly change the statement, to a claim that there is one large orbit consisting of all generating tuples of $A$. However, this will significantly complicate the proof, see~\cite{caprace-kassabov:t}. 
\end{remark}
\begin{remark}
    The last conclusion of (2) relies on the classification of finite simple groups. This dependence can be avoided when $\Aut(A)$ is much smaller than $A$, since it can be shown without using the classification that there are no non-trivial $k$-transitive groups on $n$ points for $k \gg \log n$ and $n$ sufficiently large. 
\end{remark}

It may seem that Theorem~\ref{thm:hightransitive} requires a too strong assumption --- minimality of $A$, rather than, say, simplicity. For example, in the category of associative algebras there are very few minimal algebras (since every minimal algebra is commutative). However, for the free operad $\operad_S$ as soon as $S$ contains enough operations, minimal algebras are the norm rather than the exception:
\begin{thm} \label{thm:minimal}
    Assume that $S$ contains at least two operations and that $\operad$ is free on $S$. Then, for every finite-dimensional vector space $V$ over a field $\mathbb K$, the collection of minimal $\operad$-algebra structures on $V$ is Zariski-dense among all $\operad$-algebra structures.

    In particular, for every prime $p$  the proportion of minimal algebras among all $\operad$-algebra structures on $(\Z/p)^k$ is at least $1 -6p^{(1-|S|)(k-1)}$.
\end{thm}
\begin{proof}
    Let us first write $V=\mathbb K^k$, a $k$-dimensional vector space.

    A multilinear operation $\star_s$ on $V$, of arity $\ar{s}$, is a linear map $V^{\otimes \ar{s}}\to V$, and the space of such maps has dimension $k^{\ar{s}+1}$. The set $\Sigma$ of $\operad$-algebra structures on $V$ is therefore a vector space of dimension $\sum_{s\in S}k^{\ar{s}+1}$.

    For any choice of a subspace $W\le V$, say of dimension $d$, the fact that $\star_s$ maps $W^{\otimes \ar{s}}$ back to $W$ is a linear condition imposing $d^{\ar{s}}(k-d)$ independent constraints. The subspace of $\Sigma$ consisting of algebras for which $W$ is a subalgebra therefore has codimension $\sum_{s\in S}d^{\ar{s}}(k-d)$.

    The union of all these subspaces, as $W$ varies over the Grassmann variety of $d$-dimensional subspaces, is thus a variety of codimension at least
    \[
        \sum_{s\in S}d^{\ar{s}}(k-d)-d(k-d),
    \]
    which is positive as soon as $S$ contains at least two operations.

    In the case of $\operad$-algebra structures on $(\Z/p)^k$ the above arguments show that the probability of a non-minimal structure is bounded by
    \[
        \sum_{d=1}^{k-1} \binom{k}{d}_p\; p^{-\sum_{s\in S}d^{\ar{s}}(k-d)}, 
    \]
    where the $p$-binomial coefficient $\binom{k}{d}_p=(p)_k/(p)_d(p)_{k-d}$ is the number of subspaces of $(\Z/p)^k$ of dimension $d$; here $(p)_k=(1-p)\cdots(1-p^k)$. Since all operations have arity at least $1$ and there are $|S|$ operations we have the following obvious upper bound
    \[
        \sum_{d=1}^{k-1} \binom{k}{d}_p\; p^{-|S| d(k-d)} 
    \]
    It not difficult to see that the contribution of each of the terms for $d=1$ and $d=k-1$ is bounded above by $2\frac{p}{p-1} p^{(1-|S|)(k-1)}$, which is ${}\leq 3 p^{(1-|S|)(k-1)}$ for $p>2$. 
    For all other terms we can use 
    \[
        \binom{k}{d}_p \leq \binom{k}{d} p^{d(k-d)},
    \]
    since $\binom{k}{d}_p$ counts strings $\sigma\in\{0,1\}^k$ with $d$ ones and weighted by $p^{|\{i<j:\sigma_i>\sigma_j\}|}$; this gives that the contribution of all other terms is bounded above by
    \[
        \sum_{d=2}^{k-2}\binom{k}{d}\;p^{-(|S|-1)2(k-2)},\text{ which is }{}\le (2^k-2-2k) p^{-2(|S|-1)(k-2)}\text{ if }k\ge3.
    \]
    These bounds are sufficient to prove the desired inequality for $p\geq 3$ or $k \geq 6$ and the remaining cases can be verified directly.
\end{proof}
\begin{remark}
    The probability that a random $\operad$-structure on $(\Z/p)^k$ has a $1$-dimensional subalgebra is approximately $p^{(1-|S|)(k-1)}$, so the above bound is close to optimal. It can be improved to $1 - (2+\epsilon)p^{(1-|S|)(k-1)}/(1-p^{1-|S|})$ for every $\epsilon>0$ and large enough $p$. Of course all these bounds say nothing in case $k=1$, when every algebra structure is clearly minimal.
\end{remark}

The next issue before applying Theorem~\ref{thm:hightransitive} is to show that generically the automorphism group of an $\operad_S$-algebra is very small. It is reasonable to assume that generically the only automorphisms are scalars --- a quick computation shows that $\lambda \Id$ is an automorphism of an algebra $A$ if and only if $\lambda^{\ar{s}-1} = 1$ for all $s\in S$.
Indeed this is the case:
\begin{thm}\label{thm:noauto}
    Assume that $S$ contains at least two operations. Then for any prime $p\geq 2$ most mimimal $\operad_S$-algebra structures on $(\Z/p)^k$ have ``trivial'' automorphism group, namely
    \[
        \Aut(A)  = \{ \lambda \mathsf{Id} \mid \lambda^{\ar{s}-1} = 1,\;\forall s\in S\}.
    \]
    More precisely the number of minimal algebras with non-trivial automorphism groups is less than $1/p^k$ of all possible algebra structures $p^{\sum_{s\in S} k^{\ar{s}+1} }$.
\end{thm}
\begin{proof}
    Consider $\phi\in\Aut(A)$; it is a linear map, so is given by a $k\times k$ matrix. Up to passing to a field extension, there is an eigenvector $a \in A\otimes \K$ with eigenvalue $\lambda \in \K$.

    Since $a$ generates $A\otimes \K$, we have that $A\otimes\K$ is a quotient of $\K\langle x \rangle_\operad$, so $\phi$ is uniquely determined by $a$ and $\lambda$. Moreover, since $\K\langle x \rangle_\operad$ is graded, the operator $\phi$ is diagonalizable with eigenvalues $\lambda^i$. 
    Furthermore, if all operations in $S$ have arity $1$ then $\phi$ is scalar since the whole algebra $\K\langle x \rangle_\operad$ lies in degree $1$; while if there are higher-arity operations then $\lambda$ is a root of unity. In the first case, we are done; in the second case, let $n$ be the order of $\lambda$, and for all $i\in\Z/n$ let $V_i$ be the eigenspace of $\phi$ with eigenvalue $\lambda^i$, say of dimension $d_i$.

    Let us compute the linear conditions imposed on $\phi$ by the fact that it commutes with each operation $\star_s$. It must map $V_{i_1}\otimes\cdots\otimes V_{i_{\ar{s}}}$ to $V_{i_1+\cdots+i_s}$, so the dimension of the space of $\operad$-algebra structures which commute with $\phi$ is
    \[
        \sum_{s\in S} \sum_{i_1, \dots, i_{\ar{s}}\in\Z/n} d_{i_1}\cdots d_{i_{\ar{s}}} d_{i_1+\cdots+i_{\ar{s}}}.
    \]
    Since the space $V_{i_1+\cdots+i_s}$ is not the full space, its dimension is less or equal to $k-1$.
    Therefore for each $s\in S$ we have 
    \[
        \sum_{i_1, \dots, i_{\ar{s}}\in\Z/n} d_{i_1}\cdots d_{i_{\ar{s}}} d_{i_1+\cdots+i_{\ar{s}}} \leq \sum_{i_1, \dots, i_{\ar{s}}\in\Z/n} d_{i_1}\cdots d_{i_{\ar{s}}} (k-1) = 
        k^{\ar{s} +1} - k^{\ar{s}}
    \]
    Thus the total sum is less than  
    \[
        \sum_{s\in S} k^{\ar{s}+1} -  \sum_{s\in S} k^{\ar{s}}  \leq \sum_{s\in S} k^{\ar{s}+1} -k^2 -k,
    \]
    since by assumption that there are at least $2$ operations in $S$ and one has arity at least $2$. 
    
    This shows that each candidate $\phi$ is an automorphism of a minimal algebra structure with probability at most $p^{-k^2 -k}$. Since the number of possibilities for $\phi$ is less than $p^{k^2}$, the probability that an algebra structure is minimal and has a nontrivial automorphism is less than $p^{-k}$.
\end{proof}

\begin{coro}\label{coro:main}
    Let $S$ consist of one binary operation and one operation of arity $d \geq 2$. Then the group $\Gamma_{d+2, N,\operad_S}$ has property (T) provided that $N$ is divisible by all primes less than 
    $ 3+d+4\sqrt{d-1} < 5d $. Moreover this group surjects onto $\Alt(p^{(d+2)k} -1)^{\times p^{k^{d+1}} }$ for all primes $p > 3+d+4\sqrt{d-1}$ and all $k \geq 1$.
\end{coro}
\begin{proof}
    Property (T) for the group $\Gamma_{d+2, N,\operad_S}$ is a direct consequence of Theorem~\ref{thm:t}.

    For $k \geq 2$, there are $p^{k^{d+1}}$ choices for an operation of arity $d$ on $(\Z/p)^k$. By Theorems~\ref{thm:minimal} and~\ref{thm:noauto} almost all of the operations yield minimal $\operad$-algebra structure $A$ on $(\Z/p)^k$ with trivial automorphism group. In order to count to non-isomorphic ones we need to divide by the size of the group $\GL_k(\Z/p)$. At the end it is easy to see that there are at least 
    \begin{multline*}
        p^{\sum_s k^{\ar{s}+1}} \left( 1 - 6p^{-(|S|-1)(k-1)} - p^{-k} \right) / |\GL_k(\Z/p)|  \\
        \geq p^{k^{d+1} + k^3 - k^2 } (1 - 5p^{-(d-1)(k-1)} - p^{-k} )    >  p^{k^{d+1}}
    \end{multline*}
    non-isomorphic minimal $\operad_S$ algebras structures on $(\Z/p)^k$ with trivial automorphism group.
    When $k=1$, it is easy to see that there are at least $p = p^{k^{d+1}}$ non-isomorphic $\operad_S$-structures on $\Z/p$.
    By Theorem~\ref{thm:hightransitive}(2) each of these algebras yields  a highly transitive action of $\Gamma\coloneqq\Gamma_{d+2, N,\operad_S}$ on $p^{(d+2)k} -1$ points, yielding an alternating or symmetric quotient of $\Gamma$ on that many points. Furthermore, all generators of $\Gamma$ have order $p$, which is odd by our restrictions, so this quotient is alternating. Since these actions are non isomorphic, they can be combined into a surjection from $\Gamma$ to $\Alt(p^{(d+2)k} -1)^{\times p^{k^{d+1}} }$. 
\end{proof}

\noindent Using results from~\cite{kassabov-nikolov:tau}, we can deduce:
\begin{coro}
    For every $d$ there is a group with property $(\tau)$ whose pro-finite completion is
    \[
        \prod_n \Alt(n)^{\times n^{(\log n)^d}}.
    \]
\end{coro}
\begin{proof}[Idea of the proof.]
    For any fixed $d\geq 2$, the previous construction produces a group with property (T) which maps onto $\prod_k \Alt(p^{(d+2)k} -1)^{\times p^{k^{d+1}} }$ for some fixed prime $p$. This can be combined with the results from~\cite{kassabov-nikolov:tau} to produce a group with $(\tau)$ and pro-finite completion $\prod_k \Alt(p^{(d+2)k} -1)^{\times p^{k^{d+1}} }$. Finally use that
    \[
        \prod_{n= p^{(d+2)k} -1}^{p^{(d+2)(k+1)} -2} \Alt(n)^{\times n^{(\log n)^d}}
    \]
    can be boundedly generated by $p^{d+2}$ copies of $\Alt(p^{(d+2)k} -1)^{\times p^{k^{d+1}} }$.  
\end{proof}

We do not know for which functions $f(n)$ there exists a finitely generated group with property (T) or $(\tau)$ which maps onto $\Alt(n)^{\times f(n)}$ for all $n$ --- the above construction shows that this is possible for $\log f(n) \approx (\log n)^d$ for any fixed $d$, and on the other side one needs $\log f(n) \precsim O(n \log n)$, otherwise the minimal number of generators of $\Alt(n)^{\times f(n)}$ would be unbounded.
This question is roughly equivalent to the question for which functions $f(n)$ it is possible to turn the Cayly graphs of $\Alt(n)^{\times f(n)}$ in bounded degree expanders.


\end{document}